\documentclass[11pt]{article}

\usepackage{amssymb}
\usepackage{mathrsfs}
\usepackage{amsmath}
\usepackage{amsthm}

\newcommand{\funbi}[5]
{
{#1} \left\{
\begin{array}{ll}
{#2} & \mbox{{#3}} \\
{#4} & \mbox{{#5}}
\end{array}
\right.
}

\newtheorem{theo}{Theorem}
\newtheorem{lemm}{Lemma}

\theoremstyle{plain} \newtheorem*{cond}{Presupposition of equality}
\theoremstyle{definition} \newtheorem{exam}{Example}
\theoremstyle{definition} \newtheorem{rema}{Remark}

\title{Inconsistency of Measure-Theoretic Probability\\ and\\
Random Behavior of Microscopic Systems}
\author{
\vspace{0.00cm}
Guang-Liang Li 
\footnote{Corresponding author 
(alternative e-mail address: guanglli@hotmail.com).
}
\hspace{2cm}
Victor O.K. Li
\\
Department of Electrical and Electronic Engineering\\
The University of Hong Kong \\
\{glli,vli\}@eee.hku.hk \\
}

\begin{document}

\maketitle

\begin{center}
(draft for comments)
\end{center}

%\vspace{0.2 cm}

\begin{abstract}
A mathematical theory is inconsistent if an ``equality''
is deducible in the theory, such that the two sides of
the ``equality'' are either quantities different in kind,
or quantities with different numerical values.
We report an inconsistency found in probability theory 
(also referred to as measure-theoretic probability).
For probability measures induced by real-valued random variables,
we deduce an ``equality'' such that one side of
the ``equality'' is a probability, but the other side is not.
For probability measures induced by extended random variables,
we deduce an ``equality'' such that its two sides
are unequal probabilities.
The deduced expressions are erroneous
only when it can be proved that measure-theoretic
probability is a theory free from contradiction.
However, such a proof does not exist.
The inconsistency
appears only in the theory rather than in the physical world,
and will not affect
practical applications 
as long as ideal events in the theory
(which will not occur physically) 
are not mistaken for observable events in the real world. Nevertheless,
unlike known paradoxes in mathematics,
the inconsistency cannot be explained away and hence must be resolved. 
The assumption of infinite additivity in the theory
is relevant to the inconsistency, and
may cause confusion of ideal events 
and real events.
As illustrated by an example in this article, since
abstract properties of mathematical entities in theoretical 
thinking are not necessarily properties of
physical quantities observed in the real world,
mistaking the former for the latter
may lead to misinterpreting random phenomena observed
in experiments with microscopic systems. Actually the inconsistency
is due to the notion of ``numbers'' adopted in 
conventional mathematics. A possible way to resolve the inconsistency is to treat
``numbers'' from the viewpoint of
constructive mathematics.
\end{abstract}

\section{Introduction}   
\label{sec1}                              
\hskip\parindent
A basic requirement
in mathematics and its various applications is that
the two sides of an equality must
be entities of the same kind \cite{Bena}. 
Actually, all equalities proved in a consistent
(i.e., contradiction-free) 
theory satisfy the above  requirement.
In a mathematical theory,
if an ``equality'' is deducible such that
its two sides are entities of different kinds, 
the theory is inconsistent.
No one will consider quantities of different kinds
(e.g., length and area) equal, 
even if their numerical values happen to be identical.
It is absurd even to ask whether
quantities different in kind are equal,
as they are calculated according to
different rules. 

Physical quantities (i.e., quantities in the physical world) 
are usually modeled by quantities in mathematics. For instance,
many physical quantities
are modeled by real-valued variables, functions or parameters.
Unlike physical quantities modeled by their mathematical counterparts,
real numbers themselves are pure numbers, which serve as one of 
the basic kinds of quantities in mathematics. Mathematical quantities
of different kinds follow different rules
in calculation. Operations for pure numbers are determined by the corresponding
axioms (e.g., see \cite{Royd}). In mathematics, however, not all real-valued quantities
are calculated according to the same rules. 
For example, probabilities are values assigned by
a probability measure to ``events''
modeled by measurable sets in a $\sigma$-algebra.
One cannot calculate probabilities in the same way
as one calculates values of functions which have nothing to do with
``events'', even if such functions and probability measures
have the same range. 
Similarly, probabilities and
pure numbers in the closed interval $[0, 1]$ 
are also calculated according to different rules,
and hence belong to different kinds of quantities. 
A probability describes quantitatively how likely an event will occur.
No pure number conveys such information. 
Rules of calculation with
pure numbers in the interval $[0, 1]$ are not necessarily
subject to the constraints imposed on calculating probabilities.

As shown above,
if two mathematical quantities both take real numbers as their values but
are calculated according to different rules, they
are different in kind.
Let $\Gamma$ and $\Psi$ denote two real-valued
expressions, which may be the
limits of sequences $(\Gamma_n)_{n\geq 1}$ and $(\Psi_n)_{n\geq 1}$,
respectively,
where $\Gamma_n=\Psi_n$ for each $n$, and
$\Gamma_n$ and $\Psi_n$ represent quantities of
the same kind; but the quantity represented by
$\Gamma$ (or $\Psi$) may or may not
belong to those kinds of quantities represented by
$\Gamma_n$ (or $\Psi_n$).         
For $\Gamma =\Psi$, the basic requirement mentioned
at the beginning of this section
reduces to a trite, {\em necessary} condition. 
We refer to this condition as ``presupposition of equality''
and state it below.

\begin{cond}
In a consistent mathematical theory, the two sides of $\Gamma=\Psi$
represent quantities of the same kind which are calculated based on the
same rules. The rules are determined by the kind
of quantities to which $\Gamma$ and $\Psi$ belong.
\end{cond}

The measure-theoretic foundations for
probability theory were laid by Kolmogorov in 1933 \cite{Kolm}.
Since then, it has been postulated (somewhat implicitly) that
measure-theoretic probability is a consistent theory.
However, the postulated consistency of the theory 
appears to be incompatible with the presupposition of equality.
In this article, we report an inconsistency of measure-theoretic probability.
For probability measures induced by real-valued random variables,
we deduce an ``equality'' such that one side of
the ``equality'' is a probability, 
but the other side is not.
For probability measures induced by extended random variables,
we deduce an ``equality'' such that its two sides
are unequal probabilities.
Known paradoxes in mathematics  
are either considered issues of formal 
logic (such as Russell's paradox), or treated as deducible
statements which are counterintuitive but 
true (e.g., see \cite{Czyz, Szek}), 
where ``true'' implies ``without contradiction''. 
Regarding deducible statements as 
results without contradiction is not grounded
on mathematical justification; it originates from 
a popular belief, i.e., 
mathematics is free from contradiction as it is.
The belief is an issue of mathematical philosophy (e.g., see \cite{Bena2}), 
which is not our concern.
Unlike the paradoxes treated in \cite{Czyz, Szek} or discussed by logicians,
the inconsistency found in measure-theoretic
probability cannot be explained away.

In the rest of this article, we first 
deliberate why the limit of a convergent sequence from
the interval $[0, 1]$ may not
necessarily be a limit probability, even if the terms of
the sequence are all
probabilities. 
After showing how to decide whether 
a limit probability can be calculated as
the limit of a sequence of probabilities (Section \ref{sec2}),
we review briefly the notion of tightness
of probability measures induced by random variables
(Section \ref{sec3}). Tightness is a condition to
prevent mass from ``escaping to infinity'' \cite{Bill,Poll}.
With a tight sequence of measures,
we construct a sequence of
measures which is not tight, and then
reveal the inconsistency of
measure-theoretic probability (Section \ref{sec4}).
The inconsistency is relevant to infinite additivity, which is
one of the basic assumptions in probability theory \cite{Kolm}.
Next, we clarify an ambiguity caused by the assumption
of infinite additivity (Section \ref{sec5}).
Because of this assumption,
the meaning of ``a realizable event'' is ambiguous.
Clearly, abstract properties of mathematical entities 
in theoretical thinking are not necessarily properties of
physical quantities observed in the real world. 
However, misled by the assumption of infinite additivity,
some people might mistake the former for the latter; as a result,
they might give problematic or even incorrect interpretations of
random phenomena observed in experiments with 
microscopic systems. 
We discuss this issue in the context of an
experiment in quantum physics (Section \ref{sec6}). 
Finally, we explain
why measure-theoretic probability can be so
successful even though it is not a consistent theory,
and discuss briefly how to avoid undesirable consequences 
caused by the inconsistency (Section \ref{sec7}). Actually
the inconsistency is due to the
notion of ``numbers'' adopted in conventional mathematics.
A possible way to resolve the inconsistency is to
treat ``numbers'' from the point of view of constructive mathematics
(e.g., see \cite{Brid,Rose}).

\section{Limit Probability Measure and Limit Probability}                                 
\hskip\parindent
\label{sec2}
In measure-theoretic probability, random experiments
are described by probability measures 
together with the corresponding measurable spaces.
However, when concerned with values 
and probability distributions
of random variables, we can make no
mention of the underlying measurable 
spaces associated with experiments.
Let ${\mathbb P}$ be a probability measure on a
measurable space $(\Omega, {\mathcal A})$.
Denote by ${\mathcal R}$ the $\sigma$-algebra
of Borel subsets of the real line ${\mathbb R}$ which is
equipped with the topology induced 
by the usual metric. 
Let ${\mathscr F}$ be a family of probability measures 
induced by random variables defined on $\Omega$ measurable
${\mathcal A}/{\mathcal R}$. As the probability distributions of
the random variables on $\Omega$,
members of ${\mathscr F}$ may simply be
regarded as measures defined on 
the measurable space $({\mathbb R}, {\mathcal R})$.
We do not need to specify ${\mathscr F}$ or $\Omega$. 
It suffices to know that 
random variables with their
measures belonging to ${\mathscr F}$ are 
defined on a common sample space $\Omega$. 
Not all measures on $({\mathbb R}, {\mathcal R})$ 
are members of ${\mathscr F}$, however.
Apparently, there are measures on $({\mathbb R}, {\mathcal R})$ 
induced by irrelevant random variables which are not defined on 
the sample space $\Omega$.

Let $(\rho_n)_{n\geq 1}$ be a sequence of probability measures
induced by real-valued random variables.
If $(\rho_n)_{n\geq 1}$ converges weakly,
denote by $\rho$ the weak limit
of $(\rho_n)_{n\geq 1}$ where $\rho$ is a probability measure on ${\mathcal R}$. 
Such convergence means that,
as $n\to\infty$
\[
\int_{\mathbb R}fd\rho_n\to \int_{\mathbb R}fd\rho
\]
for every
bounded, continuous function $f$ on ${\mathbb R}$ \cite{Bill,Poll}.
For probability measures $\rho, \rho_1, \rho_2,\cdots$
induced respectively by 
random variables $V, V_1, V_2, \cdots$,
if $\rho_n$ converges weakly to $\rho$,
$V_n$ is said to converge in distribution to $V$.
If random variables $U$ and $V$ have the same distribution,
$V_n$ also converges in
distribution to $U$ \cite{Poll}. The random variables
$V, V_1, V_2, \cdots$ and $U$ may be defined on
distinct probability spaces \cite{Bill}.

If a sequence $(\rho_n)_{n\geq 1}$   
converges weakly to $\rho$, and if not only
the terms $\rho_n$ but also the weak limit $\rho$ belong to ${\mathscr F}$,
we call $\rho$ the {\em limit probability measure} of $(\rho_n)_{n\geq 1}$.
For a limit probability measure $\rho$ and
an event (measurable set) $E\in{\mathcal R}$,
we call $\rho(E)$ 
a {\em limit probability}.
For every $\alpha\in{\mathscr F}$, there is a sequence
$(\alpha_n)_{n\geq 1}$ of measures from ${\mathscr F}$ such that
$\alpha_n=\alpha$ for all $n$. Hence $\alpha$ is 
the limit probability measure of $(\alpha_n)_{n\geq 1}$.
So a probability $\alpha(E)$ is trivially a limit probability.
The presupposition of equality implies that,
if one side of an equality represents 
a probability, so does the other side, and
operations performed on the
two sides must obey
the same rules for calculating probabilities.
If $\rho$ is the limit probability measure of
$(\rho_n)_{n\geq 1}$, 
we also write $\rho=[\lim_n\rho_n]$, where the brackets 
cannot be omitted. Accordingly, $\rho(E)=[\lim_n\rho_n](E)$. 
With $[\lim_n\rho_n]$ as a synonym for the limit probability measure $\rho$, 
we want to emphasize that rules for calculating
limit probabilities
and rules for calculating limits of real sequences are
different. Usually, a convergent real sequence 
from the interval $[0, 1]$
may not necessarily converge to a probability, even if
the terms of the sequence are all probabilities.
But there is an exception: the limit of a real sequence 
may ``coincide'' with a limit probability. 
We shall explain what this means shortly.

Whenever we say ``events $E_n\in{\mathcal R}$ form a sequence 
$(E_n)_{n\geq 1}$'', we assume that, for a given $E_1\in{\mathcal R}$,
there is a map $H:{\mathcal R}\mapsto{\mathcal R}$ which defines 
$E_n, n > 1$ recursively as follows:
\[
E_2=H(E_1),\;
E_3 = H(E_2)= H(H(E_1) = H^{(2)}(E_1), \;
\cdots,\;
E_n=H^{(n-1)}(E_1)
\]
and as $n\to\infty$, $E_n$ tends to a limit event $E=\lim_nH^{(n)}(E_1)$. 
An example of such $H$
is the identity map, which maps any given $E_1\in{\mathcal R}$ to
$E_1$ itself. If $H$ is not the identity map, the limit event $E$ may or may not
be a set in ${\mathcal R}$.

We are interested in measures
$\rho_n\in{\mathscr F}, n = 1, 2, \cdots$, such that
for events $E_n$ in $(E_n)_{n\geq 1}$, the 
probabilities $\rho_n(E_n)$ form a convergent 
sequence. We denote the limit of $(\rho_n(E_n))_{n\geq 1}$
by $\lim_n[\rho_n(E_n)]$ rather than
$\lim_n\rho_n(E_n)$. The meaning of
$\lim_n\rho_n(E_n)$ is ambiguous.
For instance, suppose that $E_n$ tends to a limit event $E$
and $\rho_n$ converges weakly to $\rho$ 
in a generalized sense as explained
in Remark \ref{re} (see Section \ref{sec3}).
Then $\lim_n\rho_n(E_n)$ 
may mean either $\rho(E)$ or $\lim_n[\rho_n(E_n)]$. The former
is always a probability; but the latter
may or may not be a probability.
The use of symbols ``$\lim_n[\rho_n(E_n)]$'' and
``$[\lim_n\rho_n](E)$'' may help one
avoid mixing up rules for calculating
limits of real sequences with rules for calculating
limit probabilities.
When $\lim_n[\rho_n(E_n)]$
is not a probability,
it is still 
the limit of a real sequence.

\begin{exam}
\label{eg1}
For $n=1, 2, \cdots$,
let $\delta_n$ be the Dirac measure concentrated on $E_n=\{n\}$ and let
$\rho_n = \delta_n$, where $\rho_n$ is
the probability measure induced by a degenerate random variable $W_n = n$. Since 
\[
\delta_n(\{n\})={\mathbb P}(W_n=n)=1
\]
for each $n$, one may treat
$\delta_n(\{n\}), n = 1, 2, \cdots$ as a sequence of 1's.
Such treatment implies
\[
\lim_{n\to\infty}[\delta_n(\{n\})] = 1.
\]
By definition, $\lim_n[\delta_n(\{n\})]$
is the trivial limit of a constant sequence.
Since $\delta_n$ does not converge weakly to a limit
defined on ${\mathcal R}$, $\lim_n[\delta_n(\{n\})]$
is not a limit probability.
The events $E_n, n > 1$ associated with
$\rho_n(E_n)$ are defined by $H:{\mathcal R}\mapsto{\mathcal R}$ such that
\[
E_n= H(\{n-1\})=\{n\}
\]
with $E_1=\{1\}$. Evidently, $(E_n)_{n\geq 1}$ has no limit in
${\mathcal R}$ for $\lim_nH^{(n)}(\{1\})=\{\infty\}\not\in{\mathcal R}$. 
Colloquially, one may also express this fact 
by saying that
all the mass ``escapes to infinity'' as $n\to\infty$.
$\Box$\end{exam}

When the limit event $E$ is in ${\mathcal R}$ and the limit
probability measure
$[\lim_n\rho_n]$ exists where $[\lim_n\rho_n] \in {\mathscr F}$ by definition,
we say that
$[\lim_n\rho_n](E)$ ``coincides''
with $\lim_n[\rho_n(E_n)]$, if they not only
have the same numerical value but also satisfy
the presupposition of equality, which requires
$\lim_n[\rho_n(E_n)]$ to be a probability. In other words,
by ``$[\lim_n\rho_n](E)$ and $\lim_n[\rho_n(E_n)]$ coincide''
we mean that not only

\begin{equation}
\label{coincide}
[\lim_{n\to\infty}\rho_n](E)=\lim_{n\to\infty}[\rho_n(E_n)]
\end{equation}
holds but also both sides of (\ref{coincide}) are probabilities.

\begin{rema}
Suppose that, as $n\to\infty$, three sequences different in kind, i.e.,
a sequence of measures $\rho_n\in {\mathscr F}$, a sequence of
events $E_n\in{\mathcal R}$, and a sequence of 
probabilities $\rho_n(E_n)$ approach their respective
limits: $[\lim_n\rho_n]$, a limit probability measure; $E$,
a limit event; and
$\lim_n[\rho_n(E_n)]$, the limit of a real sequence
$(\rho_n(E_n))_{n\geq 1}$. As a result,
corresponding to $\rho_n(E_n)$, there exist 
a limit of a real sequence $\lim_n[\rho_n(E_n)]$ and
a limit probability $[\lim_n\rho_n](E)$. 
If $[\lim_n\rho_n](E)$ and $\lim_n[\rho_n(E_n)]$ coincide,
the limit probability can also be calculated 
as the limit of the real sequence;
otherwise the limit probability
must be calculated as the probability assigned 
by measure $[\lim_n\rho_n]$ to event $E$.
$\Box$
\end{rema}

\begin{exam}
\label{eg2}
For $n=1, 2, \cdots$,
write $\delta_{1/n}$ for the Dirac measure concentrated on $\{1/n\}$
and let $\rho_n = \delta_{1/n}$, where $\rho_n$ is
the probability measure induced by a degenerate random variable 
$Q_n = 1/n$.
Let $H:{\mathcal R}\mapsto{\mathcal R}$ be the
identity map which defines $E_n=E_1$ for all $n > 1$
with $E_1 = (-\infty, 0]$. For each $n$,
\[
\delta_{1/n}({E_n})=\delta_{1/n}({E_1})={\mathbb P}(Q_n\in E_1) =0.
\]
By treating
$\delta_{1/n}(E_1), n = 1,2, \cdots$ as a sequence of 0's, one has
\[
\lim_{n\to\infty}[\delta_{1/n}(E_1)] = 0.
\]
The above trivial limit of the constant sequence is not
a limit probability. Actually
$[\lim_n\delta_{1/n}] = \delta_0$, 
where $\delta_0$ is the Dirac measure concentrated on $\{0\}$.
So $[\lim_n\delta_{1/n}](E_1) = 1$, which does not
coincide with $\lim_n[\delta_{1/n}(E_1)]$.
$\Box$
\end{exam}

\begin{exam}
\label{eg3}
For $E_n$ and $\rho_n$ defined in Example \ref{eg1} and
Example \ref{eg2}, respectively, when $n > 1$, 
\[
\rho_n(E_n)=\delta_{1/n}(\{n\})={\mathbb P}(Q_n =n) =0.
\]
As we already know, 
$\lim_nH^{(n)}(\{1\})=\{\infty\}\not\in{\mathcal R}$.
Clearly
\[
\lim_{n\to\infty}[\rho_n(E_n)] = 0
\]
has nothing to do with calculating a limit probability.
$\Box$
\end{exam}

The simple examples
given above involve only measures
induced by degenerate random variables. 
Actually, even for
a convergent sequence $(\rho'_n(E'_n))_{n\geq 1}$ of probabilities, 
where $\rho'_n, n = 1, 2, \cdots$ are all induced by
non-degenerate random variables, the limit of $(\rho'_n(E'_n))_{n\geq 1}$
may not necessarily be a probability, 
see Example \ref{eg5} in the next section.

Let $(\rho_n)_{n\geq 1}$ be a
sequence of measures from ${\mathscr F}$ such that
$[\lim_n\rho_n] =\rho\in {\mathscr F}$.
In other words, $\rho$ is the limit probability measure of $(\rho_n)_{n\geq 1}$.
For a given $E_1\in{\mathcal R}$, let $E_n, n>1$
be defined by 
$H:{\mathcal R}\mapsto{\mathcal R}$ recursively,
such that $\lim_nH^{(n)}(E_1) = E\in{\mathcal R}$.
Similarly, let $E_n', n>1$ be defined recursively by
$H':{\mathcal R}\mapsto{\mathcal R}$ with
$E'=\lim_nH'^{(n)}(E_1')\in {\mathcal R}$ where
$E_1'\in{\mathcal R}$ is given, and let
$(\rho_n')_{n\geq 1}$ be a sequence of measures from ${\mathscr F}$.
We have the following lemma.

\begin{lemm}
\label{le}
Assume that the three
conditions (a), (b) and (c) below are all satisfied.
(a) As $n\to\infty$,
\[
|\rho_n(E_n) - \rho_n'(E_n')|\to 0,
\]
(b) $\lim_n[\rho_n(E_n)]$ coincides with $[\lim_n\rho_n](E)$, and
(c) measure-theoretic probability is a consistent theory.
Then $\lim_n[\rho_n'(E_n')]$ is a probability.
\end{lemm}

\begin{proof}
Condition (a) implies
$\lim_n[\rho_n(E_n)] = \lim_n[\rho_n'(E_n')]$, which together with
condition (b) implies 
$\rho(E)=\lim_n[\rho_n'(E_n')]$.
The result to be proved then follows
from condition (c) and
the presupposition of equality.
\end{proof}

The above lemma is equivalent to its contrapositive. We shall use
the contrapositive of Lemma \ref{le} to reveal the inconsistency of
measure-theoretic probability.

\section{Tightness of Probability Measures}                                 
\hskip\parindent
\label{sec3}
The sequence of the Dirac measures $(\delta_n)_{n\geq 1}$ induced
by the degenerate
random variables $W_n = n$ in Example \ref{eg1}
illustrates a simple
scenario of mass ``escaping to infinity''.
Escape of mass also happens
to sequences of probability measures induced by
non-degenerate
random variables.

\begin{exam}
\label{eg4}
Let $(\mu_n)_{n\geq 1}$ be a sequence of probability measures
induced by random variables $Z_n$. 
For any given $n$, the
probability mass function of $Z_n$ is determined by $\mu_n$ as follows.
\begin{equation}
\label{Znj}
\funbi{\mu_n(\{j\})= {\mathbb P}(Z_n=j)=}
{(1-q)q^{n-j},}{$j=1, 2, \cdots, n-1$}
{1-q+q^n,}{$j=n$}
\end{equation}
where $0<q<1$. As $n\to\infty$, all the mass generated by (\ref{Znj})
will ``escape to infinity''. For any given $j < n$,
$\lim_n[\mu_n(\{j\})] = \lim_n(1-q)q^{n-j}=0$. For $j = n$,
$\mu_n(\{n\})=1-q+q^n$ is the mass at $n$.
When $j = n-k$ where
$1\leq k < n$ is fixed, $\mu_n(\{n-k\}) = (1-q)q^k$ is
the mass at $n-k$.
Similar to                                                           
the unit mass at $n$ expressed by the Dirac measure $\delta_n$,
the non-vanishing part $1-q$ of
the mass $1-q+q^n$ at $n$ and the mass $(1-q)q^k$ at $n-k$
will all ``escape to infinity'' as $n\to\infty$. So
\[
(1-q) + (1-q)\sum_{k\geq 1}q^k = 1
\]
is the total escaped mass.
$\Box$\end{exam}

\begin{exam}
\label{eg5}
As shown in the previous example,
$(\mu_n)_{n\geq 1}$ has no limit probability measure  
on ${\mathcal R}$ because of escape of mass. 
From (\ref{Znj}) we see that, for each $n$,
the probability of $\{Z_n = n\}$ is
\[
{\mathbb P}(Z_n=n)=
\mu_n(\{n\}) = 1-q + q^n.
\]
Since $(1-q+q^n)_{n\geq 1}$ is a convergent sequence,
\[
\lim_{n\to\infty}[\mu_n(\{n\})] =\lim_{n\to\infty}(1-q + q^n) = 1-q.
\]
However, as partial mass ``escaped to infinity'',
$\lim_n[\mu_n(\{n\})]$ is not a probability. Similarly,
the probability of $\{Z_n < n\}$ is
\begin{equation}
\label{PrZn<n}
{\mathbb P}(Z_n<n)=\mu_n(E_n') = 1-{\mathbb P}(Z_n=n)=
1-\mu_n(\{n\})=q-q^n
\end{equation}
where a map $H':{\mathcal R}\mapsto{\mathcal R}$ defines
\[
E_n' = H'(E_{n-1}') = E_{n-1}'\cup [n-1, n)=(-\infty, n),\; n > 1
\]
recursively with $E_1' = (-\infty, 1)$. As one can readily see,
\begin{equation}
\label{R}
E'=\lim_{n\to\infty}H'^{(n)}(E_1')={\mathbb R}\in{\mathcal R}.
\end{equation}
But
\[
\lim_{n\to\infty}[\mu_n(E_n')]
=\lim_{n\to\infty}(q-q^n)=q
\]
is also part of escaped mass rather than a probability.
$\Box$\end{exam}

Let $(\rho_n)_{n\geq 1}$ be a sequence of probability
measures defined on ${\mathcal R}$.
Tightness is a condition imposed on $(\rho_n)_{n\geq 1}$
to prevent escape of mass.
If $(\rho_n)_{n\geq 1}$ is tight, then
for each $\epsilon > 0$, there is an interval
$I_{\epsilon}\subset {\mathbb R}$,
such that $\sup_n\rho_n(I_{\epsilon}^c) < \epsilon$ where 
$I_{\epsilon}^c = {\mathbb R}\setminus I_{\epsilon}$.

\begin{exam}
\label{eg6}
Consider probability measures $\mu_n$ induced by $Z_n$ in
Example \ref{eg4}.
For each $n$, $Z_{n+1}$ takes larger values with larger probabilities.
The values of $Z_{n+1}$ include $n+1$ and all the values of $Z_n$.
The sequence $(\mu_n)_{n\geq 1}$ is not tight.
To see this,
consider $\epsilon = 1-q$. By (\ref{Znj}), 
for all $n$, $\mu_n(\{n\}) = 1-q + q^{-n} > 1-q$.
In particular, for all real numbers $b > 0$ and for all $n > N_b$, 
$\mu_n(\{n\}) > 1-q$, where $N_b$
is the smallest integer larger than or equal to $b$.
Consequently, for each interval $I_b=(-\infty, b)$, 
$\sup_n\mu_n(I_b^c) > 1-q$.
The probability mass function (\ref{Znj}) 
explains not only why $(\mu_n)_{n\geq 1}$ fails to be tight
but also where all the mass goes as shown in
Example \ref{eg4}.
$\Box$
\end{exam}

Tightness is necessary but not sufficient
for $(\rho_n)_{n\geq 1}$ to converge weakly to a limit defined on ${\mathcal R}$.
However, if $(\rho_n)_{n\geq 1}$
has a weak limit on ${\mathcal R}$, then $(\rho_n)_{n\geq 1}$
is tight.

\begin{exam}
\label{eg7}
Denote by $\lambda$ and $\lambda_n$ the probability measures 
induced by independent and identically distributed
Bernoulli random variables $X$ and $X_n, n = 1,2, \cdots$, respectively.  
The above random variables are all defined 
on a common sample space $\Omega$.
The probability mass function of $X_n$ is
\begin{equation}
\label{Xn}
\funbi{\lambda_n(\{a\})={\mathbb P}(X_n=a)=}{q,}{$a = 0$}
{1-q,}{$a=1$}
\end{equation}
where $0<q<1$. The sequence $(\lambda_n)_{n\geq 1}$ 
converges weakly to $\lambda$.
Let $E_1 = \{a\}$ and $H:{\mathcal R}\mapsto{\mathcal R}$ be
the identity map which defines $E_n = \{a\}$ for all $n > 1$. So
$\lim_nH^{(n)}(E_1)=\{a\}$. Apparently,
$[\lim_n\lambda_n](\{a\}) = \lambda(\{a\})$ which coincides with
$\lim_n[\lambda_n(\{a\})]$, and $(\lambda_n)_{n\geq 1}$ is tight.
$\Box$\end{exam}

\begin{exam}
\label{eg8}
Let $X_n, n = 1, 2, \cdots$ be the
Bernoulli random variables in Example \ref{eg7}. Denote by $Y_n$ the
largest among $X_1, X_2, \cdots, X_n$.
\begin{equation}
\label{Yn}
Y_n = \max\{X_j:j\in\{1,2,\cdots,n\}\},\;
n=1,2,\cdots.
\end{equation}
Let $\gamma_n$ be the probability measures induced by $Y_n$.
The probability mass function of $Y_n$ is
\[
\funbi{\gamma_n(\{a\})={\mathbb P}(Y_n = a)=}
{q^n,}{$a=0$}
{1-q^n,}{$a=1$.}
\]
The sequence $(\gamma_n)_{n\geq 1}$ converges
weakly to $\gamma$ where
$\gamma$ is given by
\[
\funbi
{\gamma(\{a\})=}{0,}{$a=0$}
{1,}{$a=1$}
\]
i.e., $\gamma$ is the Dirac measure $\delta_1$ concentrated on $\{1\}$.
Evidently $(\gamma_n)_{n\geq 1}$ is tight, and
$[\lim_n\gamma_n](\{a\}) = \gamma(\{a\})$ which coincides with
$\lim_n[\gamma_n(\{a\})]$.
$\Box$\end{exam}

\begin{rema}
\label{re}
By replacing ${\mathscr F}$ with a larger family 
${\overline{\mathscr F}}\supset{\mathscr F}$ 
and modifying the definition of weak convergence
accordingly, the definitions of limit probability measure
and limit probability may be generalized \cite{Loev}. 
The extended real line $\overline{\mathbb R}$
and the $\sigma$-algebra $\overline{\mathcal R}$ of subsets of $\overline{\mathbb R}$
constitute a measurable space $(\overline{\mathbb R}, \overline{\mathcal R})$.
The members of ${\overline{\mathscr F}}$ are probability measures
defined on $(\overline{\mathbb R}, \overline{\mathcal R})$ and
induced by extended random variables (on a common sample space $\Omega$)
which may take $\pm\infty$ as their values. Of course
no mass will be assigned to $\{\pm\infty\}$ by any measures in ${\mathscr F}$.
But positive probabilities can be assigned to $\{\pm\infty\}$
by measures in $\overline{\mathscr F}\setminus{\mathscr F}$.
Suppose that a sequence $(\rho_n)_{n\geq 1}$ of
measures from ${\mathscr F}$ converges weakly to a measure 
$\rho\in\overline{\mathscr F}$. 
By the {\em generalized} definitions, 
$\rho$ is the limit probability measure of $(\rho_n)_{n\geq 1}$ and
$\rho(E)$ a limit probability where $E\in \overline{\mathcal R}$,
even if $\rho\not\in{\mathscr F}$.
We still write 
$\rho = [\lim_n\rho_n]$ to avoid confusion of
limit probabilities and limits of sequences from 
the interval $[0, 1]$. 
$\Box$
\end{rema}

\begin{exam}
\label{eg9}
By the original definitions, $(\mu_n)_{n\geq 1}$ 
in Example \ref{eg5}
has no limit probability measure because of 
``mass escaping to infinity''.
However, according to the generalized definitions,
the limit probability measure $[\lim_n\mu_n]$ of
$(\mu_n)_{n\geq 1}$ is the Dirac
measure concentrated on $\{\infty\}$
induced by an extended (degenerate) random variable $Z=\infty$.
As part of mass at $\{\infty\}$,
$\lim_n[\mu_n(E_n')]$ is not
a probability. In contrast to $\lim_n[\mu_n(E_n')]$,
$[\lim_n\mu_n](E')$ is a probability
where $E' ={\mathbb R}$, see (\ref{R}). In other words,
calculated as the value of $[\lim_n\mu_n]$ at $E'={\mathbb R}$,
$[\lim_n\mu_n]({\mathbb R}) = 0$ is
the limit probability
corresponding to $\mu_n(E'_n)$.
$\Box$
\end{exam}

\section{The Inconsistency}
\hskip\parindent
\label{sec4}
With a tight sequence of
measures from ${\mathscr F}$, we can construct
a sequence of measures which is also from ${\mathscr F}$
but not tight.
The two sequences of measures are both
induced by non-degenerate random variables.

\begin{exam}
\label{eg10}
In terms of random variables $X_n$ and $Y_n$ in Example \ref{eg7} and
Example \ref{eg8}, respectively,
we may construct random variables
$Z_n$ as follows.
\begin{equation}
\label{Zn}
Z_n =\max\{j:X_j=Y_n,j\in\{1,2,\cdots,n\}\},\; n =1, 2, \cdots.
\end{equation}
For $j\leq n$, there may be more than one $j$'s
such that $X_j=Y_n$. The value of $Z_n$ is the largest of such $j$.
Let $\mu_n$ be the probability measures induced by 
$Z_n$. As we can readily see,
for $1\leq j < n$,
\[
\{Z_n = j\} = \{X_j = 1\}\cap\{X_i = 0, j < i \leq n\}
\]
and
\[
\{Z_n = n\} = \{X_n = 1\}\cup\{X_i = 0, i = 1, 2, \cdots, n\}
= \{X_n = 1\}\cup\{Y_n = 0\}.
\]
So the probability mass function of $Z_n$ is the same as (\ref{Znj})
with $q$ being the probability of $\{X_n = 0\}$ given by $\lambda_n$, see (\ref{Xn}).
From Example \ref{eg6}, $(\mu_n)_{n\geq 1}$ is not tight, though it is
constructed based on a tight sequence $(\lambda_n)_{n\geq 1}$.
$\Box$\end{exam}

As shown by the lemma below,
an ``equality'' in measure-theoretic probability can be deduced,
such that one side of the ``equality'' is a probability
but the other side is not. 
Consider $(\lambda_n)_{n\geq 1}$ and $(\mu_n)_{n\geq 1}$,
where $\lambda_n$ and $\mu_n$ are the probability measures 
induced by $X_n$ and $Z_n$, respectively (see Example \ref{eg7} and (\ref{Zn}) 
in Example \ref{eg10}).

\begin{lemm}
\label{le2}
The following expression  
\begin{equation}
\label{th-eq}
\lambda(\{0\})
=\lim_{n\to\infty}[\mu_n(E_n')]
\end{equation}
is deducible, where $\lambda=[\lim_n\lambda_n]$, $\lambda(\{0\})$ is
the probability of ``a single Bernoulli trial fails'',
and $E_n'=(-\infty, n)$.
But $\lim_n[\mu_n(E_n')]$ is not a probability.
\end{lemm}

\begin{proof}
For $n = 1, 2, \cdots$,
$\{X_n=0\} = \{Z_n<n\}\cup\{Y_n=0\}$ and
$\{Z_n<n\}\cap\{Y_n=0\}=\emptyset$,
where $Y_n$ is given by (\ref{Yn}). So
$\{X_n=0\}\setminus\{Y_n=0\} = \{Z_n<n\}$ and
\begin{equation}
\label{th-prf}
{\mathbb P}(\{X_n=0\}\setminus\{Y_n=0\})
={\mathbb P}(X_n=0)-{\mathbb P}(Y_n=0) = {\mathbb P}(Z_n<n).
\end{equation}
The left-hand side of the first equality in (\ref{th-prf})
is the probability of $\{X_n=0\}\setminus\{Y_n=0\}$, i.e.,
``the $n^{\mathrm{th}}$ trial fails but not
every trial up to the $n^{\mathrm{th}}$ one fails''.
This event is characterized by
$X_n$ and $Y_n$ so its   
probability can be
calculated by $\lambda_n$ and $\gamma_n$, i.e.,         
${\mathbb P}(X_n=0)=\lambda_n(\{0\})$ and
${\mathbb P}(Y_n=0)=\gamma_n(\{0\})$,
where $\gamma_n$ is the probability measure induced by $Y_n$.
The right-hand side of the last equality in 
(\ref{th-prf}) is the probability of $\{Z_n<n\}$, i.e.,
``for some $j < n$ the $j^{\mathrm{th}}$ trial succeeds but
the $(j+k)^{\mathrm{th}}$ trial fails 
for each $k = 1, 2, \cdots, n-j$''.
This event is
characterized by $Z_n$ and hence its probability
can be calculated by $\mu_n$,
i.e., ${\mathbb P}(Z_n<n)=\mu_n(E_n')$,                     
see (\ref{PrZn<n}).
Therefore
\begin{equation}
\label{th-prf2}
\lambda_n(\{0\})-\gamma_n(\{0\})
=\mu_n(E_n'),\;
n = 1, 2, \cdots.
\end{equation}
Letting $n\to\infty$ on both sides of (\ref{th-prf2}) gives (\ref{th-eq}).
Note that
$\lim_n[\lambda_n(\{0\})]=[\lim_n\lambda_n](\{0\})=\lambda(\{0\})$ and
$\lim_n[\gamma_n(\{0\})]=[\lim_n\gamma_n](\{0\}) =\gamma(\{0\})=0$
where $\gamma=[\lim_n\gamma_n]$
(see Examples \ref{eg7} and \ref{eg8}).
The event associated with the probability $\lambda(\{0\})$ is
$\{X=0\}$, where $X$ is a Bernoulli random variable identically
distributed as $X_n$. However, as part of mass ``escaped to infinity'',
$\lim_n[\mu_n(E_n')]$ is not a 
probability (see Examples \ref{eg4} and \ref{eg5}).
\end{proof}

A mathematical theory is inconsistent if an expression
$x=y$ is deduced in the theory, such that either (i) the expression
violates the presupposition of equality, or (ii)
$x$ and $y$ are numerically different.
The deduced expression is erroneous
only when the theory is proved to be contradiction-free
as requested by Hilbert.
According to G{\" o}del's second incompleteness theorem, 
measure-theoretic probability is one of the theories 
for which such a proof does not exist.
By Lemma \ref{le2}, (\ref{th-eq}) is deducible in measure-theoretic probability,
which implies that the theory is not consistent
as (\ref{th-eq}) violates the presupposition of equality.

\begin{theo}
\label{th2}
The mathematical theory of probability (i.e., measure-theoretic
probability) is inconsistent.
\end{theo}

\begin{proof}
Either of the two cases below 
shows that measure-theoretic probability is inconsistent.

(I) For probability measures induced by real-valued random
variables,
Lemma \ref{le} is equivalent to its contrapositive, i.e.,
if $\lim_n[\rho_n'(E_n')]$ is not a probability,
then at least one of the conditions (a),
(b) and (c) in the lemma cannot
be satisfied.
Let $\lambda_n(\{0\})$ and $\mu_n(E_n')$ in (\ref{th-prf2}) play the roles of
$\rho_n(E_n)$ and $\rho_n'(E_n')$ in Lemma \ref{le}, respectively.
By (\ref{th-prf2}) and Example \ref{eg8}, as $n\to\infty$,
\[
|\lambda_n(\{0\})-\mu_n(E_n')|=\gamma_n(\{0\})\to 0.
\]
So condition (a) is satisfied.
As shown in Example \ref{eg7}, $[\lim_n\lambda_n](\{0\})$ 
coincides with $\lim_n[\lambda_n(\{0\})]$. Thus
condition (b) is satisfied, too.
By Lemma \ref{le2}, (\ref{th-eq}) is deducible but
$\lim_n[\mu_n(E_n')]$ is not a probability.
From the contrapositive of Lemma \ref{le},
condition (c) does not hold.

(II) For probability measures induced by extended random variables,
we calculate the limit probabilities on both
sides of (\ref{th-prf2}). The limit probabilities on the
two sides of (\ref{th-prf2}) are $[\lim_n\lambda_n](\{0\})=q$
and $[\lim_n\mu_n](R)=0$ (see Remark \ref{re} and Example \ref{eg9}), respectively.
So the resultant
expression is $q=0$.
\end{proof}

To defend measure-theoretic probability,
someone might say: ``$\lambda(\{0\})$ and $\lim_n[\mu_n(E_n')]$ 
are the same pure number, and (\ref{th-eq}) is just
a tautology derived in a theory of pure numbers.''
Measure-theoretic probability is not a theory of pure numbers,
however. Evidently, $\lambda(\{0\})$ is a probability, and should not
be treated as a pure number. In Section \ref{sec1},
we have already explained why probabilities
and pure numbers are quantities of different kinds.

In the literature, it has been postulated that
measure-theoretic probability is a consistent theory.
However, the postulated consistency of measure-theoretic probability 
appears to be incompatible with the presupposition of equality.
Actually, it is the condition of
countable additivity  imposed on probability measures that 
leads to the contradiction.
By definition, the probability measure
${\mathbb P}$ is $\sigma$-additive, i.e., it
satisfies the assumption of countable additivity.
Clearly, if the restriction of countable additivity is removed,
then the inconsistency disappears.
The inconsistency revealed for the $\sigma$-additive measure 
indicates that countable additivity is relevant to the
inconsistency. 
Moreover, because of the assumption of countable additivity,
the meaning of ``a realizable event'' is ambiguous.
We shall clarify the ambiguity caused by
countable additivity
in the next section.

\section{Infinite Additivity and Ideal Events}
\hskip\parindent
\label{sec5}
For a probability measure ${\mathbb P}$ and an
infinite sequence $(B_n)_{n\geq 1}$ of mutually exclusive events with
${\mathbb P}(B_n)>0$ for all $n$, 
countable additivity requires
\[
{\mathbb P}(\cup_{n\geq 1}B_n) =\sum_{n\geq 1}{\mathbb P}(B_n)
\]
and amounts to infinite additivity. In a nontrivial way, the assumption of
infinite additivity implies that
the system of events, which includes infinitely many, mutually exclusive events
of {\em positive} probabilities,
is closed under the formation of
infinite unions and intersections. An event expressed by infinite
unions or intersections is an ideal event.
Relying on infinitely many operations 
in an infeasible
experiment, ideal events will not occur in the real world.
Nevertheless, such events may still be logically conceivable.
Since measure-theoretic probability is
sufficiently far developed, it seems that
the success of the theory might be able to justify
the assumption of infinite additivity \cite{Halm}.

No doubt measure-theoretic probability is very successful in
various applications.
However, the success of the theory 
cannot justify the assumption of infinite additivity.
For an infinite sequence $(B_n)_{n\geq 1}$ of mutually exclusive
events with
${\mathbb P}(B_n)>0$ for all $n$, the meaning of ``$B_n$ is realizable''
is ambiguous.
As long as the probability of an event is positive, no matter
how small it may be, the frequency of occurrence
of the event cannot remain zero forever, and hence
the event
will be observed at least once in an actual experiment.
In this sense 
the event is considered realizable.
But not all of
the events $B_n, n = 1, 2, \cdots$ are realizable
in the above sense, since
the number of events corresponding to {\em observable} phenomena in
any {\em actual} experiment is always finite. Therefore,
even if ${\mathbb P}(B_n) > 0$ for
all $n$ by assumption, there always exists an infinite 
subsequence of $(B_n)_{n\geq 1}$, such that
the occurrence of any event in the subsequence is only logically
conceivable but not practically observable; otherwise there would be
an instance of an infinite set observed in the real world,
which is absurd.
So the meaning of ``$B_n$ is realizable''
(as implied by ${\mathbb P}(B_n)>0$) is not the same for
all events in $(B_n)_{n\geq 1}$. In fact,
only for a finite number of events, ``realizable'' means
``practically observable in an actual experiment''.
For infinitely many events, the meaning of
``realizable'' is 
``logically conceivable but not practically observable in
any actual experiment''.
The ambiguity is due to the assumption of infinite additivity.

To avoid confusion due to the ambiguity, 
by ``a realizable (or real) event'' we mean
``an event which is practically observable in an actual experiment''.
In contrast to real events, we
call an event ``conceivable (or ideal)'' if a positive
probability is {\em assigned} to it
but the event cannot be observed in any actual experiment.
If one wants to acquire knowledge about the physical world by experiment,
it is necessary to distinguish between real events and
conceivable events; otherwise abstract properties of mathematical entities
might be mistaken for properties of the physical world.
As illustrated by the following two examples,
abstract properties
obtained by theoretical thinking are                       
not necessarily knowledge conveyed by real events.

\begin{exam}
\label{eg11}
Consider a sequence of binomial random variables $V_n, n=1,2,\cdots$ with probabilities
\[
{\mathbb P}(V_n =k)=\binom{n}{k}p_n^k(1-p_n)^{n-k}, k=0,1,\cdots, n.
\]
According to Poisson's theorem, as $n\to\infty$, 
if $p_n\to 0$ in such a way that $np_n\to c>0$, then 
\[
\lim_{n\to\infty}{\mathbb P}(V_n =k)=\frac{c^k e^{-c}}{k!},
k=0,1\cdots.
\]
The right-hand side of the above equality
is the distribution of a Poisson random variable $V$.
By definition,
$V$ takes every nonnegative integer as its value with a positive
probability. The Poisson distribution can be used to calculate approximately the
probabilities of $\{V_n=k\}, k=0,1,\cdots,n$ for the binomial random variables.
Although a positive probability is assigned to
each of $\{V=k\}, k=0,1,\cdots$ for the Poisson random variable, it is impossible for
all of such events $\{V=k\}$ to be observed in the physical world.  
$\Box$
\end{exam}

The foundation of conventional mathematics is
Zermelo-Fraenkel set theory with the axiom of choice (ZFC).
As a notion in conventional mathematics,
the set of nonnegative integers may be 
obtained by the axiom of infinity in ZFC.
However, mathematical entities such as
sets consisting of infinitely many elements do not exist
physically and cannot be observed by experiment. No instance of such
entities can be found in the real world.

\begin{exam}
\label{eg12}
Let $U$ be a continuous random variable with
range ${\mathbb R}$ and probability density function $g$
such that $g(x) > 0$ for any $x\in {\mathbb R}$.
Random variables with normal distributions are well-known
instances of such $U$.
By the properties of pure numbers, on the real line,
there are infinitely many
disjoint intervals $I_n$ of positive lengths, $n = 1, 2, \cdots$,
i.e., $\cup_{n\geq 1}I_n \subset {\mathbb R}$ and $I_n\cap I_{n'} =\emptyset$
if $n\not=n'$.
From the property of the continuous random variable, for each $n$,
\[
{\mathbb P}(U\in I_n) = \int_{x\in I_n}g(x)dx >0.
\]
However, only a finite number of $\{U\in I_n\}$
are realizable. Infinitely many of $\{U\in I_n\}$ are
unrealizable, even though a positive probability
is assigned to $\{U\in I_n\}$ for every $n$.
$\Box$
\end{exam}

From the point of view of pure mathematics, it is not
necessary for a continuous random variable to represent
any quantity in the physical world \cite{Litt}.
Actually, dividing a continuous physical quantity
into infinitely many parts is not feasible,
though it might be theoretically conceivable.
However, when modeling a continuous physical quantity
by a continuous random variable, some people might
naively believe that,
as a result of probability assignment based on
the assumption of infinite additivity,
infinitely many events,
such as those modeled by $\{U\in I_n\}, n=1,2, \cdots$ in Example \ref{eg12}, 
are not only conceivable but also realizable. 
Because of such
misunderstanding, someone might even
fail to differentiate abstract properties 
of pure numbers (which constitute a mathematical continuum)
from properties of continuous physical quantities.
We shall further discuss this issue and some other related
issues in the next section.

\section{Continuous Parameters and Microscopic Systems}
\hskip\parindent
\label{sec6}
The mathematical continuum
is a notion in pure mathematics, and cannot be
observed in realizable events.
For a continuous physical quantity modeled by a continuous
random variable, 
it is impossible to pick a {\em precise} value
out of its range {\it by measurement}. On the real line, strictly
between any two different pure numbers there is another. 
But all pure numbers are distinguishable by definition.
In general, any two different points of a topological space,
such as two pure numbers on the real line,
can be separated in mathematical reasoning, 
so long as the topological space
satisfies the separation axioms, which
are generalizations of properties  
of metric spaces (e.g., see \cite{Royd}).
However, as shown by Example \ref{eg12} in the previous section,
the {\em mathematical} properties of pure numbers 
cannot be attached to continuous {\em physical} quantities, even though
they are modeled by continuous random variables.
Consequently, there is no way
to measure continuous physical quantities
{\em arbitrarily} accurately, regardless of whether or not
the measurements are
subject to the constraints imposed by the uncertainty relations
in quantum mechanics.

For the same reason, if a continuous physical quantity is
modeled by a continuous parameter 
to be adjusted for an {\em actual}
experiment or measurement, and if values of the parameter are specified with an 
apparatus, then
the apparatus cannot yield
any precise value of the parameter. In fact,
what can be specified with the apparatus are {\em uncertain
intervals} if the parameter is real-valued, or {\em uncertain regions}
if the parameter takes its values from a metric space which may
not necessarily be the real line.
Such uncertain intervals (or regions)
are an intrinsic characteristic of modeling continuous physical quantities.
As a result,
it is impossible to repeat 
an actual experiment or measurement under
the same condition described by a precisely
specified value of a continuous parameter.

\begin{exam}
\label{eg13}
Let $a_0+\varphi$ represent a value of a continuous parameter, where
$a_0$ is a nonnegative integer and $\varphi$ in the interval $[0, 1)$.
Suppose we want to specify $\varphi$ precisely, which
amounts to defining a singleton set $\{\varphi\}$.
Every pure number $r\in [0, 1)$ has a precise expression in the
form of decimal expansion
\[
r = \sum_{n=1}^{\infty}\frac{r_n}{10^n}
=0.r_1r_2\cdots.
\]
{\em Theoretically}, $\{\varphi\}$
can be defined by specifying all of the digits
$a_1, a_2, \cdots$ of $\varphi$.
\[
\{\varphi\}=
\left \{x\in [0, 1):
x = \sum_{n=1}^{\infty}\frac{r_n}{10^n},\;
r_n = a_n, \; n = 1, 2, \cdots
\right\}.
\]
However, it is impossible to select any {\em precise} value from 
the interval $[0, 1)$
in {\em practical} applications. 
Involving {\em infinitely} many digits,
the defining property (i.e., $r_n = a_n, \; n = 1, 2, \cdots$)
of $\{\varphi\}$ cannot be used
to specify $\varphi$ for an actual experiment
or measurement.
Only the first $m$ digits of $\varphi$, i.e., 
$r_n = a_n, n = 1, 2, \cdots, m$ can be stated explicitly.
The condition $r_n = a_n$ where $1 \leq n \leq m$ is
weaker than that required by the precise expression of $\varphi$,
and defines an {\em uncertain interval}
\[
[\beta_m, \; \psi_m) =\left \{x\in [0, 1):
x = \sum_{n=1}^{\infty}\frac{r_n}{10^n},\;
r_n = a_n, \; 1 \leq n \leq m
\right\}
\]
which contains $\varphi$,
where $\beta_m = 0.a_1a_2\cdots a_m$ and
$\psi_m = 0.a_1a_2\cdots (a_m+1)$. All digits
$r_{m+k}, k=1, 2, \cdots$ of every $x\in (\beta_m, \; \psi_m)$
are unspecified.
Since $\psi_m > \beta_m$ for any positive integer $m$,
the uncertain interval
is {\em always} a set with the cardinality of the continuum, though
its length $\psi_m - \beta_m=10^{-m}$ can be very small for large $m$. 
If an actual experiment or measurement
relies on specifying values of such
a continuous parameter, only uncertain intervals in the range of the parameter
can be specified. 
$\Box$
\end{exam}

\begin{exam}
\label{eg14}
In connection with the preceding example, consider
a partial sum $\sum_{1\leq n\leq m}a_n10^{-n}$. 
One may use this partial sum as an approximation to
$\varphi$ in {\em numerical calculation}. Such an approximation also
implies selection of a precise value $\beta_m$ from the interval
$[0, 1)$. The digits in the decimal expansion of $\beta_m$
are $r_n = a_n$ for $1\leq n \leq m$ and $r_n = 0$ for $n > m$.
Just like $\varphi$, the number $\beta_m$ 
cannot be specified in {\em practical} applications.
As shown in Example \ref{eg13},
pure numbers in the interval $[0, 1)$ given
by their decimal expansions 
(or any other forms of precise expressions)
are merely {\em theoretically} conceivable;
they are useless for an
actual experiment or measurement.
No apparatus in the real world has an infinite
resolution to manipulate 
infinitely many digits. 
Setting $r = 0$ amounts to setting 
all of its digits $r_n=0, n = 1, 2, \cdots$.
Such a task is not practically achievable. 
Similarly, for the number $\beta_m$
corresponding to the partial sum 
$\sum_{1\leq n\leq m}a_n10^{-n}$, it is impossible
to put all the digits $r_n =0$ where $n = m+1,m+2,\cdots$. 
In other words, one cannot
pick $\beta_m$ out of the interval $[0, 1)$ with an apparatus. 
Since $\beta_n$ is an endpoint of
the uncertain interval
$[\beta_m, \psi_m)$, this example actually shows that, 
in practical applications, it is even
impossible to specify endpoints of intervals 
with certainty. Generally speaking, characterized by uncertain regions, 
such uncertainty in modeling continuous 
physical quantities differs from 
randomness studied by probability theory. In measure-theoretic probability, 
by assumption, the boundary of a region in a metric space 
(such as the endpoints of an interval on the real line) 
can be specified precisely; 
a value of a continuous parameter may be said to lie in
a region which is completely determined by its boundary; 
a probability may be
assigned to such a region to describe randomness
due to incapability of choosing a value
from a continuous range. However, uncertainty in
modeling continuous physical quantities cannot be captured in this way, 
since the boundary itself cannot be specified precisely.                   
$\Box$
\end{exam}

\begin{rema}
In experiments with macroscopic systems, or in thought experiments which
are not performed in the physical world, 
distinction between precise values and
uncertain regions
may be unnecessary. However, confusing
uncertain regions with
precisely specified values of continuous parameters
may lead to undesirable consequences, such as
misinterpreting random phenomena observed in {\em actual}
experiments with {\em microscopic} systems.
$\Box$
\end{rema}

\setlength{\unitlength}{0.3mm}
\begin{picture}(350,200)

\put(250,180){\line(5,-1){60}}
\put(250,180){\line(0,-1){60}}
\put(310,168){\line(0,-1){60}}
\put(250,120){\line(5,-1){60}}
\put(280,113){\vector(0,1){80}}
\put(280,97){\makebox(0,0){Polarizer A}}

\put(146,162){\line(4,-3){60}}      
\put(146,161){\line(-3,-5){30}}      
\put(206,116){\line(-3,-5){30}}   
\put(116,111){\line(4,-3){60}}      
\put(190,90){\vector(-4,3){80}}
\put(159,112){\line(0,1){60}}
\put(151,133){\makebox(0,0){$\theta$}}
\put(175,56){\makebox(0,0){Polarizer B}}

\put(330,160){\line(-4,-1){20}}  
\put(280,146){\line(-4,-1){85}}
\put(159,112){\vector(-4,-1){90}} 

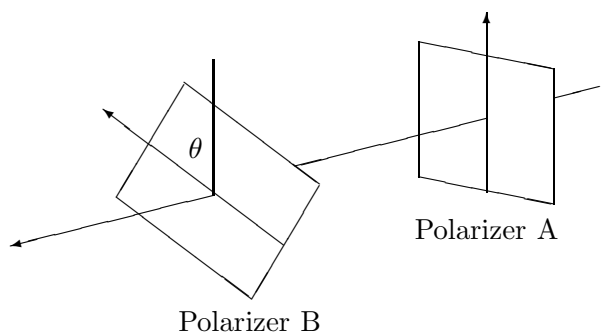
\put(220,25){\makebox(0,0){Figure 1. Experiment with
		polarized photons.}}                                    
\end{picture}

\begin{exam}
\label{eg15}
Photons in a linearly polarized beam produced by polarizer A
arrive at polarizer B. The angle (a continuous physical quantity)
between the transmission axes of the two polarizers 
is modeled by
a continuous parameter
$\theta$ (Figure 1). The intensity of the beam originating from A
can be reduced so that only one photon
at a time arrives at B. For a
specified value of $\theta$, when arriving at B,
a photon may or may not get through the polarizer.
The transmission
probability is $\cos^2\theta$ for each photon in the beam.
Since the initial conditions are considered identical for all
photons leaving A and striking B,
according to quantum mechanics,
the behavior of individual photons observed in the experiment
is ontologically stochastic (e.g., see \cite{Aule}).

The above quantum-mechanical interpretation implies an assumption:
values of a continuous parameter (i.e., $\theta$) 
can be specified precisely in the experiment. But a seemingly precisely 
specified value of $\theta$ is
actually an uncertain interval. Moreover, the direction of motion of
a photon is also
described by a continuous (vector-valued) parameter, and hence cannot be
specified precisely. A seemingly precisely specified direction of
motion for photons
in the beam is
actually a tiny 3-dimensional uncertain region 
in the range of the vector parameter. 
As a result, photons in the same beam may not necessarily 
move in the same direction. Abstract properties of mathematical
entities (such as vectors) are highly idealized descriptions of the behavior 
of their physical counterparts (such as beams of light).
For photons arriving at B,
the initial conditions described by the
continuous parameters cannot be specified precisely, and hence
are not necessarily identical.
Lack of knowledge about
the initial conditions appears relevant to the observed random
behavior of the photons. It is still controversial and even problematic to
interpret the random behavior of the photons as a consequence of natural law.
$\Box$
\end{exam}

\section{Concluding Remarks}
\hskip\parindent
\label{sec7}
Bernoulli trials have widespread applications to the world around us.
To describe Bernoulli trials in measure-theoretic probability,
one uses a sequence of probability spaces 
$(\Omega_n, {\mathcal A}_n, {\mathbb P}_n), n = 1, 2, \cdots$.
For a fixed $n, (\Omega_n, {\mathcal A}_n, {\mathbb P}_n)$ is
called a Bernoulli scheme, which 
is the probabilistic model of an experiment
consisting of $n$ independent trials, and each trial has two
possible outcomes.
Although the number of operations required by the
experiment is finite,
it is difficult to use such models when $n$ is large.
Based on the assumption of infinite additivity,
one can obtain an ideal probabilistic model
$(\Omega, {\mathcal A}, {\mathbb P})$ 
which corresponds to the limit case of the
Bernoulli schemes as $n\to\infty$. The probability space
$(\Omega, {\mathcal A}, {\mathbb P})$ describes
an ideal experiment which has infinitely many outcomes. 
This ideal experiment
consists of an infinite number of Bernoulli trials and hence
requires infinitely many operations.
As we have shown, $(\Omega, {\mathcal A}, {\mathbb P})$ implies the
inconsistency of measure-theoretic probability, and
the assumption of infinite additivity 
may result in confusion of ideal events 
and real events.

Probabilistic models of ideal experiments
are very useful for calculating probabilities
of real events approximately. Observed frequencies of occurrence
of real events can be in good agreement with results
obtained by such approximation. 
The use of ideal models in the above sense
will obviously be unobjectionable 
from an empirical point of view \cite{Kolm}.
However, conceivable events with {\em assigned} positive
probabilities should not be mixed up with
real events. For measure-theoretic probability, 
the ultimate criterion of success 
is whether predicted probabilities of {\em real} events
in various practical applications can be verified empirically.
The criterion 
implies a constraint: Since ideal events do not exist physically, 
or since it is impossible to verify the
existence of ideal events in the physical world,
physical existence of abstract properties derived by
manipulating
probabilities of ideal events
should not be inferred from experiments.
This constraint functions as
a ``firewall'' of measure-theoretic probability in applications.
When the constraint is satisfied,
it is safe to apply measure-theoretic probability to
solve problems in the real world, which may explain why
measure-theoretic probability can be so successful even if
it is not a consistent theory.
In practice, however,
violating the constraint may lead to undesirable
consequences, such as mistaking
a physically non-existing property of a mathematical notion
for an ontological reality.
A typical violation is to use 
experimental data to ``prove'' or
``verify''
{\em physical} existence of abstract 
properties of mathematical entities which have no instances in
the physical world.

The inconsistency found in measure-theoretic probability
appears merely in the theory,
and will not affect
practical applications 
as long as ideal events in the theory
are not mistaken for observable events in the real world. Nevertheless,
the inconsistency must be resolved. Physical quantities are usually
modeled by mathematical entities. If one wants
to measure or specify a continuous physical quantity as accurately as possible,
one needs to model it by a continuous quantity in mathematics.
A continuous mathematical quantity
can be divided into ``infinitely'' many parts.
But the word ``infinite'' has different meanings
in conventional mathematics and 
in constructive mathematics (e.g., see \cite{Brid,Rose}).
Measure-theoretic probability is a branch of the former.
The notion of ``infinite'' adopted in conventional mathematics
not only causes confusion of ideal events and real events,
but also leads to the inconsistency eventually. 
Because of the assumption of infinite additivity in measure-theoretic probability,
some people might mix up abstract properties of mathematical entities 
in theoretical thinking with properties of physical
quantities observed in the real world; as a result, 
they might be misguided to give problematic or even wrong 
interpretations of random phenomena observed in
experiments with microscopic systems. However,
abstract properties of mathematical entities (such as infinite sets) 
are not necessarily 
properties of their physical counterparts, and hence 
should not be attached to physical quantities.
Hopefully, the inconsistency may be resolved
by treating ``numbers'' and ``infinite'' from the
viewpoint of constructive mathematics.

\end{document}